\documentclass[12pt,a4paper]{amsart}
\usepackage{amsmath}
\usepackage{drpack}
\usepackage[english]{babel}
\usepackage{verbatim}
\usepackage[shortlabels]{enumitem}
\usepackage{color}
\usepackage{tikz-cd}

\newtheoremstyle{mio}%
{}{} 
{\itshape}{} 
{\bfseries}{.}{ } 
{#1 #2\thmnote{~\mdseries(#3)}} 
\theoremstyle{mio}
\newtheorem{teor}{Theorem}[section]
\newtheorem{cor}[teor]{Corollary}
\newtheorem{prop}[teor]{Proposition}
\newtheorem{lemma}[teor]{Lemma}
\newtheorem{defin}[teor]{Definition}

\newtheoremstyle{definition2}%
{}{} 
{}{} 
{\bfseries}{.}{ } 
{#1 #2\thmnote{\mdseries~ #3}} 
\theoremstyle{definition2}
\newtheorem{ex}[teor]{Example}

\title{Radical semistar operations}
\author{Dario Spirito}
\date{\today}
\address{Dipartimento di Scienze Matematiche, Informatiche e Fisiche, Universit\`a degli Studi di Udine, Udine, Italy}
\email{dario.spirito@uniud.it}

\newcommand{\Min}{\mathrm{Min}}

\newcommand{\inssubmod}{\mathbf{F}}
\newcommand{\inssemistar}{\mathrm{SStar}}
\newcommand{\insrad}{\mathrm{SStar_{rad}}}
\newcommand{\insspectral}{\mathrm{SStar_{sp}}}
\newcommand{\insstable}{\mathrm{SStar_{st}}}

\newcommand{\mm}{\mathfrak{m}}
\newcommand{\qspec}{\mathrm{QSpec}}
\newcommand{\psspec}{\mathrm{PsSpec}}
\newcommand{\stdstable}[1]{\widehat{#1}}
\newcommand{\deriv}{\mathcal{D}}
\newcommand{\D}{\mathcal{D}}
\newcommand{\V}{\mathcal{V}}

\begin{document}
\begin{abstract}
We introduce and study the set of radical stable operations of an integral domain $D$. We show that their set is a complete lattice that is the join-completion of the set of spectral semistar operations, and we characterize when every radical operation is spectral (under the hypothesis that $D$ is rad-colon coherent). When $D$ is a Pr\"ufer domain such that every set of minimal prime ideals is scattered, we completely classify stable semistar operations.
\end{abstract}

\maketitle

\section{Introduction}
Let $D$ be an integral domain. Semistar operations on $D$ are a class of closure operations on the set of $D$-submodules of the quotient field $K$ of $D$, defined by Okabe and Matsuda \cite{okabe-matsuda} as a generalization of the concept of star operation, originally introduced by Krull \cite{krull_breitage_I-II} and Gilmer \cite[Chapter 32]{gilmer}. Semistar operations enjoy greater flexibility than star operations, making them a good tool to use in order to study several topics relative to the properties of ideals of $D$, as well as the properties of overrings of $D$. There are several subclasses of semistar operations that are particularly of interest, among which we cite \emph{finite type} operations, \emph{eab} operations (that are related to the valuation overrings of $D$, cf. \cite{fontana_loper-eab} and \cite[Section 4]{fifolo_transactions}) and \emph{spectral} operations (related to the spectrum of $D$, cf. \cite{anderson_two_2000,anderson_intersections_2005,localizing-semistar}). See Section \ref{sect:prelim} for a precise definition.

A semistar operation $\star$ is \emph{stable} if it distributes over finite intersections, i.e., if $(I\cap J)^\star=I^\star\cap J^\star$ for all $D$-submodules $I,J$; in particular, every spectral semistar operations is stable. Stable operations are naturally connected to \emph{localizing systems} \cite{localizing-semistar} and \emph{singular length functions} \cite[Theorem 6.5 and subsequent dicussion]{length-funct}, meaning that there are canonical order-preserving bijections between the sets of stable operations, localizing systems and singular length functions on the same domain $D$ (see Section \ref{sect:other}); in particular, any classification of stable semistar operations classifies, as well, localizing systems and singular length functions. However, while spectral semistar operations can be easily classified through subsets of the spectrum of $D$ (see \cite[Remark 4.5]{localizing-semistar} and \cite[Corollary 4.4]{topological-cons}), the same does not hold for stable operations; indeed, a standard representations and a classification of stable operations have only be obtained in very specific situations, like for one-dimensional domains with scattered maximal space (see \cite{jaff-derived}) and for Pr\"ufer domains such that every ideal has only finitely many minimal primes (see \cite{stable_prufer} and \cite{length-funct}).

In this paper, we study \emph{radical} semistar operations, i.e., stable semistar operations such that, for every ideal $I$, $1\in I^\star$ if and only if $1\in\rad(I)^\star$. This notion arose in the study of almost Dedekind domains that generalize the notion of SP-domains: indeed, it can be proved that, for the class of \emph{SP-scattered domains}, every stable operation is radical \cite{almded-radfact}. We systematize the study of this class of stable operations, showing that their set $\insrad(D)$ is a complete lattice (Theorem \ref{teor:completeness}) that is, furthermore, the join-completion of the set $\insspectral(D)$ of spectral operations inside the set $\inssemistar(D)$ of all semistar operations. For \emph{rad-colon coherent domains} (a large class of domain that includes domains with Noetherian spectrum and Pr\"ufer and coherent domains), it follows that the set $\insrad(D)$ depends uniquely on the spectrum of $D$ (in the sense that any two such domains with homeomorphic spectra have isomorphic set of radical operations; Theorem \ref{teor:insrad-iso}).

In Section \ref{sect:scattered}, we connect the study of radical operations with the use of the derived set and of scattered topological spaces (following \cite{jaff-derived,almded-radfact,PicInt}) to show that (under the hypothesis that $D$ is rad-colon coherent) the two sets $\insrad(D)$ and $\insspectral(D)$ coincide if and only if the space $\Min(I)$ of minimal ideals of $D$ is scattered for every ideal $I$. Specializing further to the case of Pr\"ufer domain, we show that this property is enough to obtain a full classification of all stable operations of $D$ by means of a standard representation (Theorem \ref{teor:prufer-MinI}), generalizing the results obtained in \cite{stable_prufer} and \cite{length-funct} for the case where each $\Min(I)$ is finite; in particular, we show that for these Pr\"ufer domains the set $\insstable(D)$ depends only on the spectrum of $D$ (as a topological space) and on which prime ideals are locally principal (Theorem \ref{teor:prufer-iso}). In particular, these results hold when the spectrum of $D$ is countable.

In Section \ref{sect:other}, we define the concepts analogue to radical semistar operations in the context of localizing systems and length functions.

\section{Preliminaries}\label{sect:prelim}
Throughout the paper, $D$ will denote an integral domain with quotient field $K$, and $\inssubmod(D)$ will denote the set of $D$-submodules of $K$. An \emph{overring} of $D$ is a ring between $D$ and $K$.

\subsection{Semistar operations}
A \emph{semistar operation} on $D$ is a map $\star:\inssubmod(D)\longrightarrow\inssubmod(D)$, $I\mapsto I^\star$, such that, for every $I,J\in\inssubmod(D)$, $x\in K$:
\begin{itemize}
\item $I\subseteq I^\star$;
\item if $I\subseteq J$, then $I^\star\subseteq J^\star$;
\item $(I^\star)^\star=I^\star$;
\item $(xI)^\star=x\cdot I^\star$.
\end{itemize}
A submodule $I$ is said to be \emph{$\star$-closed} if $I=I^\star$. The set of $\star$-closed ideals uniquely determines $\star$.

The set $\inssemistar(D)$ of the semistar operations on $D$ has a natural partial order, where $\star_1\leq\star_2$ if and only if $I^{\star_1}\subseteq I^{\star_2}$ for every $I\in\inssubmod(D)$, or equivalently if every $\star_2$-closed ideal is $\star_1$-closed. Under this order, $\inssemistar(D)$ is a complete lattice: the infimum of a family $\{\star_\alpha\}_{\alpha\in A}$ is the map
\begin{equation*}
I\mapsto\bigcap_{\alpha\in A}I^{\star_\alpha},
\end{equation*}
while its supremum is the semistar operation $\sharp$ such that a submodule $I$ is $\sharp$-closed if and only if it is $\star_\alpha$ closed for every $\alpha\in A$.

An ideal $I$ of $D$ is said to be a \emph{quasi-$\star$-ideal} if $I=I^\star\cap D$; if $I$ is a prime quasi-$\star$-ideal, we say that $I$ is a \emph{quasi-$\star$-prime}. The set of quasi-$\star$-primes is called the \emph{quasi-spectrum} of $\star$, and is denoted by $\qspec^\star(D)$.

A semistar operation $\star$ is said to be \emph{of finite type} if $I^\star=\bigcup\{J^\star\mid J\subseteq I$ is finitely generated$\}$, for every $I\in\inssubmod(D)$. It is \emph{semi-finite} (or \emph{quasi-spectral}) if every quasi-$\star$-ideal is contained in a quasi-$\star$-prime; every semistar operation of finite type is semi-finite.

A very general way to define semistar operations is through overrings: any family $\Lambda$ of overrings induces the semistar operation
\begin{equation*}
\star_\Lambda:I\mapsto\bigcap_{T\in\Lambda}IT.
\end{equation*}
When $\Lambda$ is a family of localizations of $D$, we say that $\star$ is a \emph{spectral} semistar operation. A spectral semistar operation can also be defined through a subset of the spectrum $\Spec(D)$ of $D$: given a family $\Delta\subseteq\Spec(D)$, we denote by $s_\Delta$ the semistar operation
\begin{equation*}
s_\Delta:I\mapsto\bigcap_{P\in\Delta}ID_P.
\end{equation*}
Setting $\Delta^\downarrow:=\{Q\in\Spec(D)\mid Q\subseteq P$ for some $P\in\Delta\}$, we have that $\qspec^{s_\Delta}(D)=\Delta^\downarrow$, and that $s_\Delta=s_{\Delta^\downarrow}$; moreover, $s_\Delta=s_\Lambda$ if and only if $\Delta^\downarrow=\Lambda^\downarrow$ \cite[Remark 4.5]{localizing-semistar}. A spectral operation $s_\Delta$ is of finite type if and only if $\Delta$ is compact, with respect to the Zariski topology \cite[Corollary 4.4]{topological-cons}.

A semistar operation is \emph{stable} if $(I\cap J)^\star=I^\star\cap J^\star$ for every $I,J\in\inssubmod(D)$. Every spectral semistar operation is stable, while every semi-finite stable operation is spectral \cite[Theorem 4]{anderson_overrings_1988}. There exist stable operations that are not spectral: an example is the \emph{$v$-operation} $I\mapsto(D:(D:I))$ when $D$ is a valuation domain with non-principal maximal ideal. The \emph{localizing system associated} to $\star$ is \cite[Section 2]{localizing-semistar}
\begin{equation*}
\mathcal{F}^\star:=\{I\text{~ideal of~}D\mid I^\star\cap D=D\}=\{I\text{~ideal of~}D\mid 1\in I^\star\};
\end{equation*}
this set uniquely determines $\star$, in the sense that if $\star_1,\star_2$ are stable, then $\star_1=\star_2$ if and only if $\mathcal{F}^{\star_1}=\mathcal{F}^{\star_2}$. More precisely, $\star_1\leq\star_2$ if and only if $\mathcal{F}^{\star_1}\subseteq\mathcal{F}^{\star_2}$.

The set $\insspectral(D)$ of spectral semistar operations of $D$ is closed by infimum, but not by supremum (see \cite[Example 4.5]{spettrali-eab} and Example \ref{ex:supnonrad} below); note, however, that $\insspectral(D)$ is a complete lattice (see below the discussion after Corollary \ref{cor:MinIfinito}). On the other hand, the set $\insstable(D)$ of stable operations is closed by both infima and suprema \cite[Proposition 5.3]{non-ft}.

\subsection{Topologies on the spectrum}
Let $\Spec(D)$ denote the spectrum of $D$, i.e., the set of all prime ideals of $D$. We denote by $\V(I)$ and $\D(I)$, respectively, the closed and the open sets of the Zariski topology associated to an ideal $I$; i.e., $\V(I):=\{P\in\Spec(D)\mid I\subseteq P\}$, while $\D(I):=\Spec(D)\setminus\V(I)$.

The spectrum of a ring can also be endowed with two other topologies. The \emph{inverse} topology is the topology whose subbasic open sets are those in the form $\V(I)$, as $I$ ranges among the finitely generated ideals of $D$; the \emph{constructible} topology is the topology whose subbasic open sets are the $\V(I)$ and the $\D(I)$, for $I$ ranging among the finitely generated ideals of $D$. In particular, the constructible topology is finer than both the Zariski and the inverse topology, and, furthermore, it is Hausdorff.

If $I$ is an ideal of $D$ and $\Min(I)$ denotes the set of minimal primes of $D$, the Zariski and the constructible topology agree on $\Min(I)$ (by \cite[Corollary 4.4.6(i)]{spectralspaces-libro}, applied to the spectral space $\V(I)$).

\subsection{Derived sets and scattered spaces}
Let $X$ be a topological space. A point $x\in X$ is \emph{isolated} if $\{x\}$ is an open set; the set of non-isolated points of $X$ is called the \emph{derived set} of $X$, and is denoted by $\deriv(X)$. Given an ordinal $\alpha$, we define the \emph{$\alpha$-th derived set} as
\begin{equation*}
\deriv^\alpha(X):=\begin{cases}
\deriv(\deriv^\gamma(X)) & \text{if~}\alpha=\gamma+1;\\
\bigcap_{\beta<\alpha}\deriv^\beta(X) & \text{if~}\alpha\text{~is a limit ordinal}.
\end{cases}
\end{equation*}
If $\deriv^\alpha(X)=\emptyset$ for some $\alpha$, the space $X$ is said to be \emph{scattered}. On the other hand, if $\deriv(X)=X$, then $X$ is said to be \emph{perfect}.

\section{Radical semistar operations}
\begin{defin}
We say that a semistar operation $\star$ on $D$ is \emph{quasi-radical} if, whenever $1\notin I^\star$ for some ideal $I$ of $D$, then $1\notin\rad(I)^\star$.
\end{defin}

We collect in the next few propositions the main properties of quasi-radical semistar operations.
\begin{prop}\label{prop:psrad-Dstar}
Let $D$ be an integral domain and $\star$ be a semistar operation on $D$. If $\star|_{\inssubmod(D^\star)}$ is quasi-radical as a semistar operation on $D^\star$, then $\star$ is quasi-radical.
\end{prop}
\begin{proof}
Let $I$ be an ideal of $D$ such that $1\notin I^\star$. Then, $1\notin(ID)^\star=(ID^\star)^\star$, and thus, by hypothesis, $1\notin\rad(ID^\star)^\star$. However, $\rad(I)\subseteq\rad(ID^\star)$; hence, $1\notin\rad(I)^\star$. It follows that $\star$ is quasi-radical.
\end{proof}

\begin{prop}\label{prop:pseudrad-ex}
Let $D$ be an integral domain and $\star$ be a semistar operation on $D$.
\begin{enumerate}[(a)]
\item If $\star$ is semi-finite, then it is quasi-radical.
\item If $\star$ is of finite type, then it is quasi-radical.
\item If $\star$ is induced by overrings, then it is quasi-radical.
\item If $\star$ is spectral, then it is quasi-radical.
\end{enumerate}
\end{prop}
\begin{proof}
Suppose $\star$ is semi-finite, and let $I$ be an ideal of $D$ such that $1\notin I^\star$. Then, $J:=I^\star\cap D$ is a quasi-$\star$-ideal such that $1\notin J^\star$. Since $\star$ is semi-finite, there is a quasi-$\star$-prime ideal $P$ containing $J$; thus, $1\notin P^\star\supseteq\rad(J)^\star\supseteq\rad(I)^\star$. Therefore, $\star$ is quasi-radical.

The next three points follows from the fact that every semistar operation of finite type is semi-finite, as well as any semistar operation induced by overrings, and that any spectral semistar operation is induced by overrings.
\end{proof}

\begin{prop}\label{prop:quasirad-inf}
Let $\{\star_\alpha\}_{\alpha\in A}$ be a set of quasi-radical semistar operations on $D$. Then, $\inf_{\alpha\in A}\star_\alpha$ is quasi-radical.
\end{prop}
\begin{proof}
Let $\star:=\inf_{\alpha\in A}\star_\alpha$, and let $I$ be an ideal of $D$ such that $1\notin I^\star$. Since $I^\star=\bigcap_{\alpha\in A} I^{\star_\alpha}$, it follows that there is a $\beta\in A$ such that $1\notin I^{\star_\beta}$. Since $\star_\beta$ is quasi-radical, then $1\notin\rad(I)^{\star_\beta}$, and thus also $1\notin\rad(I)^\star$. Hence $\star$ is quasi-radical.
\end{proof}

The previous proposition does not extend to the supremum of a family of quasi-radical operations, as the next example shows.
\begin{ex}
Let $D$ be a Pr\"ufer domain of dimension $1$ such that $\Max(D)=\{P,Q_0,Q_1,\ldots,Q_n,\ldots,\}$ is countable and with a single non-isolated point, $P$; suppose also that $D_P$ is not discrete. For every $n\inN$, let $T_n:=\bigcap_{i\geq n}D_{Q_i}$; then, $T_n$ is a Pr\"ufer domain whose maximal ideals are the extensions of $Q_i$ (for $i\geq n$) and of $P$; in particular, $\bigcup_nT_n=D_P$.

Recall that a \emph{fractional ideal} of a domain $T$ is an $I\in\inssubmod(T)$ such that $dI\subseteq T$ for some $d\in K$, $d\neq 0$. For every $n$, let $\sharp_n$ and $\star_n$ be the semistar operations defined by
\begin{equation*}
I^{\sharp_n}:=\begin{cases}
IT_n & \text{if~}IT_n\text{~is a fractional ideal over~}T_n\\
K & \text{otherwise},
\end{cases}
\end{equation*}
and
\begin{equation*}
I^{\star_n}:=I^{\sharp_n}\cap(ID_P)^{v_P},
\end{equation*}
where $v_P$ is the $v$-operation on $D_P$. Since $T_n\subseteq D_P$, for every ideal $I$ of $D$ we have $I^{\star_n}=I^{\sharp_n}$: hence, if $1\notin I^{\star_n}$ then $IT_n\neq T_n$ and so $\rad(I)T_n\neq T_n$. Thus, every $\star_n$ is quasi-radical.

Let now $\star$ be the supremum of all $\star_n$. Then, $D_P^\star=D_P$ since $D_P^{\star_n}\subseteq(D_P)^{v_P}=D_P$. Moreover, if $t\in D_P$, then $t\in T_n$ for some $n$, and thus $t\in D^{\star_n}\subseteq D^\star$. Hence, $D^\star=D_P$. For every $n$, $PD_P$ is not a fractional ideal over $T_n$, and thus
\begin{equation*}
(PD_P)^{\star_n}=K\cap(PD_P)^{v_P}=K\cap D_P=D_P.
\end{equation*}
Hence,
\begin{equation*}
P^\star=(PD)^\star=(PD^\star)^\star=(PD_P)^\star=D_P.
\end{equation*}
On the other hand, if $L\neq P$ is a $P$-primary ideal, then $(LD_P)^{v_P}=LD_P$; hence, $LD_P$ is $\star$-closed and thus $L^\star\subseteq LD_P\cap D$, so that $1\notin L^\star$ while $1\in P^\star=\rad(L)^\star$. Therefore, $\star$ is not quasi-radical.
\end{ex}

The main problem of the previous example is that the restriction of a quasi-radical operation on $D$ to an overring of $D$ is not quasi-radical (as it happens for $\star_i|_{\inssubmod(D_P)}$); this in turn is due to the fact that the property of being quasi-radical depends only on the ideals of $D$, rather than on all $D$-submodules of $K$. For this reason, we are only interested in the following subclass of semistar operations.

\begin{defin}
We say that a semistar operation $\star$ on $D$ is \emph{radical} if it is quasi-radical and stable.
\end{defin}

\begin{lemma}\label{lemma:overring-radical}
Let $\star$ be a radical stable operation, and suppose that $T$ is an overring of $D$. Then $\star|_{\inssubmod(T)}$ is radical.
\end{lemma}
\begin{proof}
Let $I$ be a $T$-ideal such that $1\notin I^\star$. Then, $1\notin(I\cap D)^\star$, and since $\star$ is radical we have $1\notin(\rad(I\cap D))^\star$. However, $\rad(I\cap D)=\rad(I)\cap D$; hence
\begin{equation*}
1\notin\rad(I\cap D)^\star=(\rad(I)\cap D)^\star=\rad(I)^\star\cap D^\star.
\end{equation*}
Thus $1\notin\rad(I)^\star$ and so $\star|_{\inssubmod(T)}$ is radical, as claimed.
\end{proof}

\begin{prop}\label{prop:radJchiuso}
Let $D$ be an integral domain and let $\star$ be a radical semistar operation on $D$ such that $D=D^\star$. Let $J$ be an ideal of $D$ such that $J=J^\star$. Then, $\rad(J)^\star=\rad(J)$.
\end{prop}
\begin{proof}
Let $s\in\rad(J)^\star$, and let $t\in s^{-1}\rad(J)\cap D$. Then, $st\in\rad(J)$, and thus there is an $n$ such that $s^nt^n\in J$, i.e., $t^n\in s^{-n}J\cap D$. Hence $t\in\rad(s^{-n}J\cap D)$ and so $s^{-1}\rad(J)\cap D\subseteq\rad(s^{-n}J\cap D)$.

Since $s\in\rad(J)^\star$, we have $1\in s^{-1}\rad(J)^\star$; hence also $1\in \rad(s^{-n}J\cap D)^\star$. Since $\star$ is radical, it follows that $1\in (s^{-n}J\cap D)^\star$; thus $1\in s^{-n}J^\star$ and $s^n\in J^\star=J$. Therefore, $s\in\rad(J)$, and $\rad(J)^\star=\rad(J)$.
\end{proof}

\begin{teor}\label{teor:radical-compllattice}
Let $D$ be an integral domain. Then, the set $\insrad(D)$ of radical stable semistar operations is a complete sublattice of $\inssemistar(D)$.
\end{teor}
\begin{proof}
Let $\{\star_\alpha\}_{\alpha\in A}$ be a family of radical semistar operations. Then, its infimum is quasi-radical by Proposition \ref{prop:quasirad-inf} and stable since every $\star_\alpha$ is stable, and thus $\insrad(D)$ is closed by infima. Let $\star$ be the supremum of $\{\star_\alpha\}_{\alpha\in A}$.

Let $T:=D^\star$: then, $T$ is $\star_\alpha$-closed for every $\alpha$. By Proposition \ref{prop:psrad-Dstar}, it suffices to show that $\star|_{\inssubmod(T)}$ is radical; furthermore, by Lemma \ref{lemma:overring-radical}, each $\star_\alpha|_{\inssubmod(T)}$ is radical. Therefore, without loss of generality, we can actually suppose that $T=D$, i.e., that $D$ is $\star_\alpha$-closed for every $\alpha$.

Let $J$ be an ideal of $D$ such that $1\notin J^\star$. Let $L:=J^\star$; then, $L$ is an ideal of $D$ that is $\star_\alpha$-closed for every $\alpha$, and thus by Proposition \ref{prop:radJchiuso} also $\rad(L)$ is $\star_\alpha$-closed for every $\alpha$; thus, $\rad(L)=\rad(L)^\star$. In particular, $1\notin\rad(L)^\star$; the claim now follows from the fact that $\rad(J)\subseteq\rad(L)$.
\end{proof}

\section{Radical operations as a completion}
By Proposition \ref{prop:pseudrad-ex}, each spectral semistar operation $s_\Delta$ is radical; in this section, we explore the link between these two classes of semistar operations. Following \cite[Example 4.5]{spettrali-eab}, we first give an example of a radical operation that is not spectral.
\begin{ex}\label{ex:supnonrad}
Let $\ins{A}$ be the ring of all algebraic integer, i.e., the integral closure of $\insZ$ in $\overline{\insQ}$. Then, $\ins{A}$ is a B\'ezout domain (every finitely generated ideal is principal) and, for every maximal ideal $P$, we have that $\ins{A}=\bigcap\{\ins{A}_Q\mid Q\in\Max(\ins{A})\setminus\{P\}\}$. Hence, for each $P$ the spectral operation $\sharp(P):=s_{\Max(\ins{A})\setminus\{P\}}$ closes $\ins{A}$, and thus the supremum $\star$ of all the $\sharp(P)$ closes $\ins{A}$ too, and thus it closes every principal ideal (since $(x\ins{A})^\star=x\cdot\ins{A}^\star=x\cdot\ins{A}$).

As the supremum of a family of radical operations, $\star$ is itself radical. However, for every $P$-primary ideal $Q$, we have $Q^{\sharp(P)}=\ins{A}$; therefore, $\qspec^\star(D)$ contains only the zero ideal. In particular, were $\star$ spectral, it should be equal to $s_{(0)}$, and in particular $1$ would belong to $I^\star$ for every nonzero ideal $\star$, contradicting the fact that principal ideals are closed. Hence $\star$ is radical, but not spectral.
\end{ex}

The following proposition characterizes which radical operations are spectral.
\begin{prop}\label{prop:caratt-spectral}
Let $\star$ be a radical stable operation on $D$. Then, $\star$ is spectral if and only if, for every radical ideal $I$,
\begin{equation*}
I^\star\cap D=\bigcap\{P\mid P\in \V(I)\cap\qspec^\star(D)\}.
\end{equation*}
\end{prop}
\begin{proof}
Suppose first that $\star$ is spectral, say $\star=s_\Delta$ with $\Delta=\Delta^\downarrow$. For every $P\in\Delta$, the ideal $ID_P$ is radical, and its minimal primes are the minimal primes of $I$ contained in $P$; all of them belong to $\Delta$, and thus they are all in $\V(I)\cap\qspec^\star(D)$. Hence,
\begin{equation*}
I^\star\cap D=\bigcap_{P\in\Delta}\{QD_P\cap D\mid Q\in\Min(I),Q\subseteq P\}=\bigcap_{Q\in\Min(I)\cap\Delta}Q.
\end{equation*}
The claim follows.

Conversely, suppose that the equality holds, and let $\Delta:=\qspec^\star(D)$. For every $P\in\Delta$, $PD_P$ is $\star$-closed, and thus $\star$ is the identity on $\inssubmod(D_P)$; it follows that $I^\star\subseteq ID_P$ for every $P\in\Delta$, and thus $\star\leq s_\Delta$.

Suppose that $\star<s_\Delta$: then, there is an ideal $I$ of $D$ such that $I^\star\subsetneq I^{s_\Delta}$. Let $x\in I^{s_\Delta}\setminus I^\star$ and let $J:=(I:_Dx)$. Since $\star$ is stable, we have $1\in J^{s_\Delta}$ while $1\notin J^\star$; since both $s_\Delta$ and $\star$ are radical, it follows that $1\in\rad(J)^{s_\Delta}$ while $1\notin\rad(J)^\star$. However, by the hypothesis and the first part of the proof, $\rad(J)^{s_\Delta}\cap D=\rad(J)^\star\cap D$; this is a contradiction, and thus $\star$ must be equal to $s_\Delta$. In particular, $\star$ is spectral, as claimed.
\end{proof}

\begin{cor}\label{cor:MinIfinito}
Let $D$ be an integral domain such that every ideal has only finitely many minimal primes. Then, every radical stable operation is spectral.
\end{cor}
\begin{proof}
Let $I$ be a radical ideal and $P_1,\ldots,P_n$ be its minimal primes. Then, $I=P_1\cap\cdots\cap P_n$, and thus $I^\star=P_1^\star\cap\cdots\cap P_n^\star$. Since $\star$ is stable, for each $i$ the ideal $P_i^\star\cap D$ is either equal to $P_i$ or to $D$ \cite[Lemma 3.1]{stable_prufer} hence, $I^\star\cap D$ is equal to the intersection of the minimal primes that are quasi-$\star$-ideals. By Proposition \ref{prop:caratt-spectral}, $\star$ is spectral.
\end{proof}

The following result is a variant of \cite[Lemma 3.1]{stable_prufer}.
\begin{prop}\label{prop:chius-rad}
Let $\star$ be a stable semistar operation, and let $J$ be a radical ideal of $D$. Then, $J^\star\cap D$ is either $D$ or a radical ideal.
\end{prop}
\begin{proof}
Suppose $J^\star\cap D\neq D$. Let $s\in D$ be such that $s^n\in J^\star$ for some integer $n$. Let $L:=s^{-n}J\cap D$: since $\star$ is stable, $1\in L^\star$. We claim that $s^{-1}J\cap D=\rad(L)$. Indeed, if $x\in s^{-1}J\cap D$ then $sx\in J$ and thus also $s^nx\in J$, i.e., $x\in s^{-n}J\cap D=L\subseteq\rad(L)$. On the other hand, if $x\in\rad(L)$, then $x^k\in s^{-n}J$ for some $k$, and thus $x^ks^n\in J$. Since $x,s\in D$, we have $x^Ns^N\in J$, where $N:=\max\{n,k\}$; since $J$ is radical, it follows that $xs\in J$, that is, $x\in s^{-1}J\cap D$. Thus $s^{-1}J\cap D=L=\rad(L)$.

Since $1\in L^\star$, it follows that $1\in(s^{-1}J)^\star=s^{-1}J^\star$, that is, $s\in J^\star$. Hence $J^\star$ is radical, as claimed. 
\end{proof}

\begin{prop}\label{prop:dense-closed}
Let $D$ be a domain, let $I$ be a radical ideal of $D$ and $\Delta=\Delta^\downarrow\subseteq\Spec(D)$. Then, the following are equivalent:
\begin{enumerate}[(i)]
\item $I=I^{s_\Delta}\cap D$;
\item $\V(I)\cap\Delta$ is dense in $\V(I)$;
\item $\Min(I)\cap\Delta$ is dense in $\Min(I)$.
\end{enumerate}
\end{prop}
\begin{proof}
By Proposition \ref{prop:chius-rad}, the ideal $J:=I^{s_\Delta}\cap D$ is a radical ideal of $D$ containing $I$; therefore, $\V(J)\subseteq \V(I)$ is a closed set, and $\V(J)$ contains $\V(I)\cap\Delta$ since, if $P\in\V(I)\cap\Delta$, then $J=I^\star\cap D\subseteq P^\star\cap D=P$. In particular, if $\V(I)\cap\Delta$ is dense then it must be $I=J$. On the other hand, if $\V(I)\cap\Delta$ is not dense, then there is an ideal $L$ such that $\Delta\cap \V(L)\subsetneq \V(I)$; thus, $ID_P=LD_P$ for every $P\in\Delta$, and $J\supseteq L$, so that $I\neq J$. Thus, the first two conditions are equivalent.

If $\Min(I)\cap\Delta$ is dense in $\Min(I)$, then $\Min(I)$ is contained in the closure of $\V(I)\cap\Delta$; then, $\V(I)\cap\Delta$ is dense since $\Min(I)$ is dense in $\V(I)$. Conversely, suppose $\V(I)\cap\Delta$ is dense in $\V(I)$ and take $P\in\Min(I)$. For every open set $\Omega$ meeting $\V(I)$, $\Omega\cap\Delta\cap\V(I)$ is nonempty; if $Q$ belongs to the intersection, then $\Omega\cap\Delta\cap\Min(I)$ contains the minimal primes of $I$ contained in $Q$. Hence, $\Delta\cap\Min(I)$ is dense in $\Min(I)$. Thus also the last two conditions are equivalent.
\end{proof}

Let $D$ be an integral domain. The space $\insspectral(D)$ of spectral semistar operation on $D$ is a complete lattice: indeed, let $X:=\{s_{\Delta_\alpha}\mid\alpha\in A\}$ be a subset of $\insspectral(D)$ with $\Delta_\alpha=\Delta_\alpha^\downarrow$. Then, setting $\Delta^\cup:=\cup_\alpha\Delta_\alpha$ and $\Delta^\cap:=\bigcap_\alpha\Delta_\alpha$, it is easy to see that the infimum of $X$ in $\insspectral(D)$ is $s_{\Delta^\cup}$ and that its supremum is $s_{\Delta^\cap}$.

However, while $s_{\Delta^\cup}$ is also the infimum of $X$ as a subset of $\inssemistar(D)$, the same does not hold for $s_{\Delta^\cap}$ (see Example \ref{ex:supnonrad}). We now want to prove that the set $\insrad(D)$ of radical semistar operations is the join-completion of $\insspectral(D)$ in $\inssemistar(D)$. In particular, the construction of Example \ref{ex:supnonrad} is the only way to obtain non-spectral radical semistar operations.

\begin{teor}\label{teor:completeness}
Let $D$ be an integral domain. Then:
\begin{enumerate}[(a)]
\item $\insspectral(D)$ is join-dense in $\insrad(D)$;
\item $\insrad(D)$ is the completion of $\insspectral(D)$ in $\inssemistar(D)$.
\end{enumerate}
\end{teor}
\begin{proof}
Since $\insrad(D)$ is a complete sublattice of $\inssemistar(D)$ (Theorem \ref{teor:radical-compllattice}), we only need to prove that every radical stable operation is the supremum of a family of spectral operations.

Fix thus $\star\in\insrad(D)$. Let $\Delta\subseteq\Spec(D)$ be such that $\Delta\cap  \V(I)$ is dense in $\V(I)$ for every radical ideal $I$ such that $I=I^\star\cap D$. Then, $s_\Delta\leq\star$: indeed, if $J$ is an ideal such that $1\in J^{s_\Delta}$ and $1\notin J^\star$, then $\Delta\cap \V(J^\star\cap D)$ would be dense in $\V(J^\star\cap D)$, and thus by Proposition \ref{prop:dense-closed} $J^\star\cap D$ would be quasi-$s_\Delta$-closed, against the fact that $1\in J^{s_\Delta}$. Hence, $s_\Delta\leq\star$. Let $\sharp$ be the supremum of all such $s_\Delta$: by construction, $\sharp\leq\star$.

We claim that $\star=\sharp$. Let $J$ be a proper radical ideal: if $1\in J^\sharp$, then $1\in J^\star$ since $\sharp\leq\star$. Suppose that $1\in J^\star$. We claim that $\D(J)\cap \V(I)$ is dense in $\V(I)$ for every radical ideal $I$ such that $I=I^\star\cap D$. If not, there is a $P\in \V(J)$ that is not in the closure of $\D(I)\cap \V(J)$; hence, there is a radical ideal $L$ such that $P\in \D(L)$ and $\D(L)\cap \D(J)\cap \V(I)=\emptyset$. Since $\D(L)\cap \D(J)=\D(L\cap J)$, it follows that $\D(L\cap J)\cap \V(I)=\emptyset$, and thus $\V(I)\subseteq \V(L\cap J)$. Thus, $L\cap J\subseteq I$, and $L^\star\cap J^\star=(L\cap J)^\star\subseteq I^\star$. Hence
\begin{equation*}
(J^\star\cap D)\cap(L^\star\cap D)\subseteq I^\star\cap D=I.
\end{equation*}
By hypothesis, $J^\star$ contains $1$; hence, $J^\star\cap D=D$ and $L^\star\cap D\subseteq I$, so that $L\subseteq I$. In particular, $\D(L)\subseteq \D(I)$; it follows that $\D(L)\cap \V(I)=\emptyset$, against the hypothesis that $P\in \D(L)\cap \V(I)$. Therefore, $\D(J)\cap \V(I)$ is dense in $\V(I)$ for all radical ideal $I$ such that $I=I^\star\cap D$; thus, $s_{\D(J)}$ is one of the spectral operations used to define $\sharp$; hence, $s_{\D(J)}\leq\sharp$. It follows that $1\in J^{s_{\D(J)}}\subseteq J^\sharp$. Therefore, $1\in J^\star$ if and only if $1\in J^\sharp$; since $\star$ and $\sharp$ are stable, it follows that $\star=\sharp$, as claimed, and $\star$ is in the completion of $\insspectral(D)$.
\end{proof}

\section{Isomorphic sets of radical operations}
Let $D_1,D_2$ be two integral domains. If $\phi:\Spec(D_1)\longrightarrow\Spec(D_2)$ is an order isomorphism, then $\phi$ induces an order isomorphism $\Phi:\insspectral(D_1)\longrightarrow\insspectral(D_2)$ by setting $\Phi(s_\Delta)=s_{\phi(\Delta)}$ for every $\Delta\subseteq\Spec(D_1)$. However, $\Phi$ does not, in general, extend to a similar isomorphism between the set of radical semistar operations, for example because it may be $\insspectral(D_1)=\insrad(D_1)$ while $\insspectral(D_2)\neq\insrad(D_2)$ (take for example $D_1:=K[X]$ and $D_2:=\ins{A}$, where $K$ is a field of the same cardinality of $\Max(\ins{A})$).

In this section, we extend this result to radical operations by using the Zariski topology. We work in a particular class of domains: we say that a domain is \emph{rad-colon coherent} if, for every $x\in K$, the radical of the ideal $(D:_Dx)$ is the radical of a finitely generated ideal. This property is linked with the relationship between the Zariski, inverse and constructible topology of $\Spec(D)$ and the Zariski, inverse and constructible topology of $\Over(D)$. Every Noetherian domain (or, more generally, every domain with Noetherian spectrum) is rad-colon coherent; likewise, every Pr\"ufer domain and every coherent domain are rad-colon coherent, as well as every polynomial in finitely many variables over a Pr\"ufer domain. See \cite{localizzazioni} for applications of this property and for an example of a domain that is not rad-colon coherent.

In our context, the reason why we use this notion is essentially the following lemma.
\begin{lemma}\label{lemma:rcc-Tstar}
Let $D$ be a rad-colon coherent domain and let $I$ be a radical ideal. Define $T:=\bigcap\{D_P\mid P\in\Min(I)\}$. If $\star$ is a radical semistar operation such that $I=I^\star\cap D$, then $T^\star=T$ and $(IT)^\star=IT$.
\end{lemma}
\begin{proof}
Suppose first that $\star=s_\Delta$ is spectral, with $\Delta=\Delta^\downarrow$. Then, by Proposition \ref{prop:dense-closed}, $\Delta\cap \V(I)$ is dense in $\V(I)$ and $\Delta\cap\Min(I)$ is dense in $\Min(I)$, with respect to the Zariski topology. By \cite[Corollary 4.4.6(i)]{spectralspaces-libro}, the Zariski and the constructible topology agree on $\Min(I)$; hence, $\Delta\cap\Min(I)$ is dense in $\Min(I)$ also with respect to the constructible topology.

Let $x\in T^\star$, and let $J:=(D:_Dx)=x^{-1}D\cap D$. We claim that $\V(J)\cap\Min(I)\cap\Delta=\emptyset$. Indeed, let $P\in\Min(I)\cap\Delta$. Since $x\in T^\star\subseteq D_P^\star$, we have $1\in(x^{-1}D_P)^\star$, and thus $1\in(JD_P)^\star$; however, if $P\in \V(J)$ then $(JD_P)^\star\subseteq(PD_P)^\star=PD_P$ since $P\in\Delta$. Therefore, $\V(J)\cap\Min(I)\cap\Delta=\emptyset$, and thus $\Min(I)\cap\Delta\subseteq \D(J)$. Since $D$ is rad-colon coherent, $\rad(J)$ is the radical of a finitely generated ideal, and thus $\D(J)$ is a closed subset, with respect to the constructible topology; thus $\D(J)\cap\Min(I)$ is closed in $\Min(I)$. Since $\Min(I)\cap\Delta$ is dense in $\Min(I)$, it follows that $\D(J)\cap\Min(I)$ must be equal to the whole of $\Min(I)$, that is, $\V(J)\cap\Min(I)=\emptyset$. Thus, $JD_P=D_P$ for every $P\in\Min(I)$, and $x\in T$. Hence, $T^\star=T$.

This also implies that $(IT)^\star$ is a radical ideal of $T$ contained in $PT$ for every $P\in\Min(I)$. Hence $(IT)^\star=IT$, as claimed.

Suppose now that $\star$ is any radical operation. By Theorem \ref{teor:completeness}, $\star$ is the supremum of a family $Y$ of spectral semistar operation. For each $\sharp\in Y$, we have $\sharp\leq\star$, and thus $I=I^\sharp\cap D$; by the previous part of the proof, $T^\sharp=T$ and $(IT)^\sharp=IT$. Hence, also $T^\star=T$ and $(IT)^\star=IT$, as claimed.
\end{proof}

\begin{prop}\label{prop:sup-rcc}
Let $D$ be a rad-colon coherent domain and let $I$ be a radical ideal. Let $Y$ be a family of radical semistar operations and let $\sharp:=\sup Y$. If $I=I^\star\cap D$ for every $\star\in Y$, then $I=I^\sharp\cap D$.
\end{prop}
\begin{proof}
Let $T:=\bigcap\{D_P\mid P\in\Min(I)\}$. By Lemma \ref{lemma:rcc-Tstar}, $(IT)^\star$ is closed by every $\star\in Y$, and thus also $(IT)^\sharp$ is closed. Then, $I^\sharp\cap D\subseteq(IT)^\sharp\cap D=I$, and thus $I=I^\sharp\cap D$.
\end{proof}

We are ready to prove the main result of this section.
\begin{teor}\label{teor:insrad-iso}
Let $D_1,D_2$ be rad-colon coherent integral domains, and suppose that there is a homeomorphism $\phi:\Spec(D_1)\longrightarrow\Spec(D_2)$. Then, there is an order isomorphism
\begin{equation*}
\Phi:\insrad(D_1)\longrightarrow\insrad(D_2)
\end{equation*}
such that $\Phi(s_\Delta)=s_{\phi(\Delta)}$ for every $\Delta\subseteq\Spec(D_1)$.
\end{teor}
\begin{proof}
Let $X_i:=\insspectral(D_i)$ and $Y_i:=\insrad(D_i)$ for $i=1,2$.

By Theorem \ref{teor:radical-compllattice}, $Y_1$ is a join-completion of $X_1$; hence, we can consider $Y_1$ as a sublattice of the set $\mathcal{L}(X_1)$ of lower sets of $X_1$ by the map $\epsilon_1$, defined by $\epsilon_1(y)=\{x\in X_1\mid x\leq y\}$ for every $y\in Y$. In particular, $\epsilon_1(x)=\{x\}^\downarrow$ for every $x\in X_1$. Likewise, we can consider $Y_2$ as a sublattice of $\mathcal{L}(X_2)$ through a map $\epsilon_2$ defined analogously.

The map
\begin{equation*}
\begin{aligned}
\Phi\colon X_1 & \longrightarrow X_2,\\
s_\Delta& \longmapsto s_{\phi(\Delta)}
\end{aligned}
\end{equation*}
is an order isomorphism; thus, it can be extended to a map $\widetilde{\Phi}$ between $\mathcal{L}(X_1)$ and $\mathcal{L}(X_2)$, that remains an order isomorphism. We claim that $\widetilde{\Phi}(\epsilon_1(Y_1))=\epsilon_2(Y_2)$, and to do so it is enough to prove that, if $A\subseteq X_1$, then the supremum $\sup_{Y_1}A$ in $Y_1$ (that is, the supremum of $A$ as a semistar operation) is spectral if and only if $\sup_{Y_2}\Phi(A)$ is spectral.

Suppose first that $\star:=\sup_{Y_1}A$ is not spectral, and let $\sharp=s_\Delta$ be the supremum of $A$ in $X_1$. Let $\star'$ and $\sharp'$ be, respectively, the supremum of $\Phi(A)$ in $Y_2$ and $X_2$. By construction, $\star<\sharp$, and thus there is a radical ideal $I$ such that $I=I^\star\cap D_1$ while $I^\sharp=D_1^\sharp$. Let now $J$ be the radical ideal such that $\V(J)=\phi(\V(I))$; we claim that $J=J^{\star'}\cap D$ while $J^{\sharp'}=D_2^{\sharp'}$.

Indeed, if $s_\Lambda\in\Phi(A)$, then $\phi^{-1}(\Lambda)\cap\V(I)$ is dense in $\V(I)$, and thus $\Lambda\cap\V(J)$ is dense in $\V(J)$; since $D_2$ is rad-colon coherent, by Proposition \ref{prop:sup-rcc} $J=J^{\sup\Phi(A)}\cap D_2$, i.e., $J=J^{\star'}\cap D_2$. On the other hand, $\sharp=s_\Delta$ for some $\Delta$ such that $\Delta\cap\V(I)=\emptyset$; hence, $\sharp'=\Phi(\sharp)=\Phi(s_\Delta)=s_{\phi(\Delta)}$, where $\phi(\Delta)\cap\V(J)$ is empty. Hence, $J^{\sharp'}=D_2^{\sharp'}$. Thus, $\star'\neq\sharp'$, and $\sup_{Y_2}\Phi(A)$ is not spectral. 

The opposite implication follows by applying the same reasoning to the homeomorphism $\phi^{-1}$ (which induces the map $\Phi^{-1}$ on the sets of spectral semistar operations).

Therefore, $\widetilde{\Phi}$ restricts to an isomorphism between $\epsilon_1(Y_1)$ and $\epsilon_2(Y_2)$; since $Y_i\simeq\epsilon_i(Y_i)$ for $i=1,2$, it follows that $Y_1=\insrad(D_1)$ and $Y_2=\insrad(D_2)$ are isomorphic, as claimed.
\end{proof}

\section{When every spectral operation is radical}\label{sect:scattered}
We have seen that, in general, not every radical semistar operation is spectral, although the two sets are equal when every ideal has only finitely many minimal primes (Corollary \ref{cor:MinIfinito}). In this section, we characterize when the two sets are equal for rad-colon coherent domains; specializing to Pr\"ufer domain, we also show that under this hypothesis we can obtain a standard representation of all stable operations.

We start with two  topological lemmas.
\begin{lemma}
Let $D$ be an integral domain and $I$ a radical ideal that is not prime. Then, $\Min(I)$ is not perfect if and only if there are a prime ideal $Q$ and a radical ideal $J\neq I$ such that $I=Q\cap J$.
\end{lemma}
\begin{proof}
If $I$ is not perfect, there is an isolated point $Q$ of $\Min(I)$, and $\Min(I)\setminus\{Q\}=\Min(I)\cap\V(J)$ for some radical ideal $J$. By construction, $J\supsetneq I$ and $\V(J)\cup\V(Q)=\V(I)$, so that $I=Q\cap J$. Conversely, if $I=Q\cap J$, then $\V(I)=\V(Q)\cup\V(J)$. Since $I\neq J$, $\V(J)$ cannot contain all minimal primes of $I$; therefore, $Q$ must be contained in $\Min(I)$. Hence, $\{Q\}=\Min(I)\setminus\V(J)$ is open in $\Min(I)$ and $Q$ is isolated; thus $\Min(I)$ is not perfect.
\end{proof}

\begin{lemma}\label{lemma:Min-scat-perf}
Let $D$ be an integral domain.  Then, the following are equivalent:
\begin{enumerate}[(i)]
\item\label{lemma:Min-scat-perf:scat} $\Min(I)$ is scattered for every ideal $I$;
\item\label{lemma:Min-scat-perf:perf} $\Min(I)$ is not perfect for every ideal $I$.
\end{enumerate}
\end{lemma}
\begin{proof}
\ref{lemma:Min-scat-perf:scat} $\Longrightarrow$ \ref{lemma:Min-scat-perf:perf} is obvious. To show \ref{lemma:Min-scat-perf:perf} $\Longrightarrow$ \ref{lemma:Min-scat-perf:scat}, let $I$ be a radical ideal and let $X:=\bigcap_\alpha\deriv^\alpha(\Min(I))$: then, $X$ is perfect. Let $J:=\bigcap\{Q\mid Q\in X\}$; then, $I\subseteq J\subseteq P$ for all $P\in X$, and thus $X\subseteq\Min(J)$. We claim that $\Min(J)$ is perfect. Indeed, suppose not: then, it has an isolated point $P$, and $P$ cannot belong to $X$, since $X$ is perfect. Since $P$ is isolated, there is a finitely generated ideal $L$ such that $\D(L)\cap\Min(J)=\{P\}$; therefore, $L\nsubseteq P$ while
\begin{equation*}
L\subseteq\bigcap_{Q\in\Min(J)\setminus\{P\}}\subseteq\bigcap_{Q\in X}Q=J,
\end{equation*}
a contradiction. Thus $\Min(J)$ is perfect, as claimed.
\end{proof}

\begin{defin}
We say that $D$ is \emph{min-scattered} if $\Min(I)$ is a scattered space for every ideal $I$.
\end{defin}

\begin{prop}\label{prop:countable}
Let $D$ be a domain such that $\Spec(D)$ is countable. Then, $D$ is min-scattered.
\end{prop}
\begin{proof}
The space $\Spec(D)$, endowed with the constructible topology, is Hausdorff, compact and countable, and thus scattered  \cite{mazur-sierp-numerabili}. Therefore, for every ideal $I$, also $\Min(I)$ is scattered, with respect to the constructible topology; however, on each $\Min(I)$ the constructible and the Zariski topology coincide. Hence $D$ is min-scattered.
\end{proof}

\begin{teor}\label{teor:rcc-radspectral}
Let $D$ be a rad-colon coherent domain. Then, the following are equivalent:
\begin{enumerate}[(i)]
\item\label{teor:rcc-radspectral:scat} $D$ is min-scattered;
\item\label{teor:rcc-radspectral:radical} every radical semistar operation is spectral.
\end{enumerate}
\end{teor}
\begin{proof}
\ref{teor:rcc-radspectral:radical} $\Longrightarrow$ \ref{teor:rcc-radspectral:scat} Suppose that there is a radical ideal $I$ such that $\Min(I)$ is perfect. For every $P\in\Min(I)$, let $\sharp(P):=s_{\Min(I)\setminus\{P\}}$, and let $\star$ be the supremum of all these $\sharp(P)$. Then, $\star$ is a radical semistar operation; we claim that $\star$ is not spectral.

Indeed, since $\Min(I)$ is perfect each $\Min(I)\setminus\{P\}$ is dense, and thus by Proposition \ref{prop:caratt-spectral} $I=I^{\sharp(P)}\cap D$ for every $P$; since $D$ is rad-colon coherent, by Proposition \ref{prop:sup-rcc} we have $I=I^\sharp\cap D$. However, $P^{\sharp(P)}\ni 1$ for every $P\in\Min(I)$; hence, $1\in P^\sharp$ for every $P\in \V(I)$. By Proposition \ref{prop:caratt-spectral}, $\sharp$ cannot be spectral.

\ref{teor:rcc-radspectral:scat} $\Longrightarrow$ \ref{teor:rcc-radspectral:radical} Suppose that there is a radical operation $\star$ that is not spectral. By Proposition \ref{prop:caratt-spectral}, there is an ideal $I$ such that $I^\star\cap D\subsetneq\bigcap\{P\mid P\in \V(I)\cap\qspec^\star(D)\}$; without loss of generality we can suppose that $I=I^\star$. Let $J$ be equal to the intersection, and let $\Gamma:=\Min(I)\setminus \V(J)$. By construction, $\Gamma$ is nonempty.

The set $\Gamma$ does not contain isolated points of $\Min(I)$: if $Q\in\Gamma$ is isolated, then $I=Q\cap I_0$ for some $I_0\supsetneq I$, and thus $I^\star\cap D=(Q\cap I_0)^\star\cap D=Q^\star\cap I_0^\star\cap D$ can only be equal to $I$ if $Q=Q^\star\cap D$, i.e., $Q\in\qspec^\star(D)$ and $J\subseteq Q$.

For every $P\in\Gamma$, let $\gamma(P)$ be the minimal ordinal number such that $P\notin\deriv^\gamma(\Gamma)$. Note that $\gamma(P)$ exists since $\Min(I)$ is scattered and no element of $\Gamma$ is isolated; furthermore, $\gamma(P)$ is a successor ordinal. Let $\gamma$ be the minimal element of the set of all $\gamma(P)$, and let $Q\in\Min(I)$ be such that $\gamma(Q)=\gamma$. Let also $\beta$ be such that $\gamma=\beta+1$.

By construction, $Q$ is a limit point of $\Min(I)$, while $Q$ is isolated in $\deriv^\beta(\Min(I))$. Hence, $Q$ is a limit point of $\Min(I)\setminus\deriv^\beta(\Min(I))$. The latter set is contained in $\V(J)$, by definition of $Q$; since $\V(J)$ is closed, it follows that also $Q\in \V(J)$. This is a contradiction, and thus $\Gamma$ must be empty, i.e., there cannot be a radical non-spectral semistar operation. The claim is proved.
\end{proof}

We now restrict to the case of Pr\"ufer domains, extending results proved in \cite{stable_prufer} and mostly following the general method of that paper. Given a semistar operation $\star$ on the Pr\"ufer domain $D$, we define the \emph{pseudo-spectrum} $\psspec^\star(D)$ as the set of those prime ideals $Q$ such that $1\in Q^\star$, but there is a $Q$-primary ideal $L$ such that $L=L^\star\cap D$. Using the quasi-spectrum and the pseudo-spectrum, we can define from $\star$ a new semistar operation $\stdstable{\star}$, called the \emph{normalized stable version} of $D$, as
\begin{equation*}
\stdstable{\star}:I\mapsto\bigcap_{P\in\qspec^\star(D)}ID_P\cap\bigcap_{Q\in\psspec^\star(D)}(ID_Q)^{v_Q},
\end{equation*}
where $v_Q$ is the $v$-operation on the valuation domain $D_Q$.  Note that $v_P$ is different from the identity on $D_P$ if and only $P$ is idempotent.

\begin{lemma}\label{lemma:radical-stdstable}
Let $\star$ be a radical semistar operation. Then, $\star=\stdstable{\star}$ if and only if $\star$ is spectral.
\end{lemma}
\begin{proof}
If $\star$ is spectral, then $\star=s_{\qspec^\star(D)}=\stdstable{\star}$. Conversely, if $\star=\stdstable{\star}$ but $\star$ is not spectral, there is a $Q\in\psspec^\star(D)$. By definition, $1\in Q^\star$, while $1\notin L^\star$ for some $Q$-primary ideal $L$; since $\rad(L)=Q$, this contradicts the fact that $\star$ is radical. Hence $\star$ must be spectral.
\end{proof}

\begin{teor}\label{teor:prufer-MinI}
Let $D$ be a Pr\"ufer domain. Then, the following are equivalent:
\begin{enumerate}[(i)]
\item\label{teor:prufer-MinI:scat} $D$ is min-scattered;
\item\label{teor:prufer-MinI:radical} every radical semistar operation is spectral;
\item\label{teor:prufer-MinI:stdstable} $\star=\stdstable{\star}$ for every stable semistar operation $\star$.
\end{enumerate}
\end{teor}
\begin{proof}
\ref{teor:prufer-MinI:scat} $\iff$ \ref{teor:prufer-MinI:radical} follows from Theorem \ref{teor:rcc-radspectral}, since a Pr\"ufer domain is rad-colon coherent, while \ref{teor:prufer-MinI:stdstable} $\Longrightarrow$ \ref{teor:prufer-MinI:radical} follows from Lemma \ref{lemma:radical-stdstable}.

To prove \ref{teor:prufer-MinI:scat} $\Longrightarrow$ \ref{teor:prufer-MinI:stdstable}, fix a stable semistar operation $\star$. By \cite[Theorem 3.9]{stable_prufer}, $\star\leq\stdstable{\star}$, and thus if $1\in I^\star$ then also $1\in I^{\stdstable{\star}}$. Suppose that $1\in I^{\stdstable{\star}}$ while $1\notin I^\star$. Then, $J:=I^\star\cap D$ is a proper ideal of $D$ that is quasi-$\star$-closed. Changing notation from $I$ to $J$, we can suppose without loss of generality that $I=I^\star\cap D$.

Since $\Min(I)$ is not perfect, there is an isolated point $Q$. Since $\Min(I)\setminus\{Q\}$ is closed, it is equal to $\V(I_1)\cap\Min(I_1)$ for some radical ideal $I_1$. Let $T:=\bigcap\{D_P\mid P\in\V(Q)\}$ and $S:=\bigcap\{D_P\mid P\in\V(I_1)\}$: Then, $I=IS\cap IT$, and in particular
\begin{equation*}
I^\star\cap D=(IS)^\star\cap D\cap(IT)^\star\cap D.
\end{equation*}
The radical of $IT\cap D$ is $Q$, which is a prime ideal. By the proof of \cite[Theorem 4.5]{stable_prufer}, since $1\in (IT\cap D)^{\stdstable{\star}}$, we also have $1\in(IT\cap D)^\star$, and thus $(IT)^\star\cap D=D$. On the other hand, $IS\cap D$ is not contained in $Q$; hence, neither does $(IS)^\star\cap D$. By construction, $I=I^\star\cap D$; this is a contradiction, and thus $\star$ and $\stdstable{\star}$ must be equal, as claimed.
\end{proof}

\begin{cor}
Let $D$ be a domain such that $\Spec(D)$ is countable. Then, every radical semistar operation is spectral. If $D$ is Pr\"ufer, moreover, $\star=\stdstable{\star}$ for every stable semistar operation $\star$.
\end{cor}
\begin{proof}
If $\Spec(D)$ is countable, then $D$ is min-scattered by Proposition \ref{prop:countable}. The claims now follow from Theorems \ref{teor:rcc-radspectral} and \ref{teor:prufer-MinI}.
\end{proof}

The following is a version of Theorem \ref{teor:insrad-iso} for stable operations on a Pr\"ufer domain; it can also be seen as a variant of \cite[Theorem 5.12]{length-funct} (in view of \cite[Section 6]{length-funct}).
\begin{teor}\label{teor:prufer-iso}
Let $D_1,D_2$ be Pr\"ufer domains. Suppose that there is a homeomoprhism $\phi:\Spec(D_1)\longrightarrow\Spec(D_2)$ such that a prime ideal $P$ is idempotent if and only if $\phi(P)$ is idempotent. If $D_1$ is min-scattered, then there is an isomorphism $\Phi:\insstable(D_1)\longrightarrow\insstable(D_2)$.
\end{teor}
\begin{proof}
We first note that if $J$ is an ideal of $D_2$, then $\Min(J)$ is the set of minimal elements of the closed set $\V(J)$; then, $\phi^{-1}(\V(J))$ is a closed set of $\Spec(D_1)$, and thus it is equal to $\V(J')$ for some ideal $J'$ of $D_1$. By hypothesis, $\Min(J')$ is scattered, and thus also $\phi(\Min(J'))=\Min(J)$ is scattered. Hence also $D_2$ is min-scattered.

Given a stable semistar operation $\star$ on $D_1$, we define $\Phi(\star)$ as the map
\begin{equation*}
\Phi(\star):I\mapsto\bigcap_{P\in\phi(\qspec^\star(D_1))}I(D_2)_P\cap\bigcap_{Q\in\phi(\psspec^\star(D_1))}(I(D_2)_Q)^{v_Q}.
\end{equation*}
We claim that $\qspec^{\Phi(\star)}(D_2)=\phi(\qspec^\star(D_1))$ and $\psspec^{\Phi(\star)}(D_2)=\phi(\psspec^\star(D_1))$.

Indeed, let $P\in\Spec(D_1)$ and let $Q:=\phi(P)$. If $P\in\qspec^\star(D_1)$, then
\begin{equation*}
Q^{\Phi(\star)}\cap D\subseteq Q(D_2)_Q\cap D_2=Q,
\end{equation*}
and thus $Q\in\qspec^{\Phi(\star)}(D_2)$; conversely, if $Q\in\qspec^{\Phi(\star)}(D_2)$, then either $Q(D_2)_A\neq (D_2)_A$ for some $A\in\phi(\qspec^\star(D_1))$ or $(Q(D_2)_B)^{v_B}\neq Q(D_2)_B$ for some $B\in\phi(\psspec^\star(D_1))$. In the former case, $P\subseteq A$; since the quasi-spectrum is closed by generizations \cite[Proposition 3.4(a)]{stable_prufer}, $P\in\qspec^\star(D_1)$ and thus $Q\in\phi(\qspec^\star(D_1))$. In the latter case it must be $Q\subsetneq B$, and thus $Q\in\phi(\qspec^\star(D_1))$ since every $B'\subsetneq B$ is in the quasi-spectrum  \cite[Proposition 3.4(b)]{stable_prufer}. Therefore, $\qspec^{\Phi(\star)}(D_2)=\phi(\qspec^\star(D_1))$.

Suppose now that $P\in\psspec^\star(D_1)$. By the previous paragraph, $Q\notin\qspec^{\Phi(\star)}(D_1)$. There is a $P$-primary ideal $L\subsetneq P$ such that $L=L^\star\cap D$; in particular, $P$ is not branched in the valuation domain $(D_1)_P$. The map $\phi$ induced a homeomorphism between $\Spec((D_1)_P)$ and $\Spec((D_2)_Q)$; therefore, neither $Q$ is branched, and thus there exist a $Q$-primary ideal $L'\subsetneq Q$. By definition,
\begin{equation*}
(L')^{\Phi(\star)}\cap D_2\subseteq L'(D_2)_Q\cap D_2=L',
\end{equation*}
and thus $Q\in\psspec^{\Phi(\star)}(D_2)$. Conversely, if $Q\in\psspec^{\Phi(\star)}(D_2)$, then there is a $Q$-primary ideal $L\subsetneq Q$ such that $L^{\Phi(\star)}\cap D_2=L$. By the previous part of the proof, $P\notin\qspec^\star(D_1)$; if $P$ is not even in $\psspec^\star(D_1)$, then $L^{\Phi(\star)}$ would just be equal to $D^{\Phi(\star)}$, a contradiction. Hence, $Q=\phi(P)\in\phi(\psspec^\star(D_1))$. Therefore, $\psspec^{\Phi(\star)}(D_2)=\phi(\psspec^\star(D_1))$.

Consider now the map $\Psi:\insstable(D_2)\longrightarrow\insstable(D_1)$ defined by
\begin{equation*}
\Psi(\sharp):I\mapsto\bigcap_{P\in\phi^{-1}(\qspec^\sharp(D_2))}I(D_1)_P\cap\bigcap_{Q\in\phi^{-1}(\psspec^\sharp(D_2))}(I(D_1)_Q)^{v_Q}
\end{equation*}
for every ideal $I$ of $D_1$ and every $\sharp\in\insstable(D_2)$. Then, $\Psi$ is the map associated to the homeomorphism $\phi^{-1}$ by the previous construction; hence,
\begin{align*}
\Psi\circ\Phi(\star):I& \mapsto \bigcap_{P\in\phi^{-1}(\qspec^{\Phi(\star)}(D_2))}I(D_1)_P\cap\bigcap_{Q\in\phi^{-1}(\psspec^{\Phi(\star)}(D_2))}(I(D_1)_Q)^{v_Q}=\\
&= \bigcap_{P\in\qspec^\star(D_1)}I(D_1)_P\cap\bigcap_{Q\in\psspec^\star(D_1)}(I(D_1)_Q)^{v_Q}=I^{\stdstable{\star}}
\end{align*}
since $\phi$ is a homeomorphism. By Theorem \ref{teor:prufer-MinI}, $\stdstable{\star}=\star$, and thus $\Psi\circ\Phi(\star)=\star$, i.e., $\Psi\circ\Phi$ is the identity on $\insstable(D_1)$. By symmetry, also $\Phi\circ\Psi$ is the identity on $\insstable(D_2)$; hence, $\Phi$ and $\Psi$ are inverses one of each other. It is straightforward to see that they are also order-preserving; thus, they establish a isomorphism between $\insstable(D_1)$ and $\insstable(D_2)$, as claimed.
\end{proof}

\begin{cor}
Let $D_1,D_2$ be Pr\"ufer domains with countable spectrum. Suppose that there is a homeomoprhism $\phi:\Spec(D_1)\longrightarrow\Spec(D_2)$ such that a prime ideal $P$ is idempotent if and only if $\phi(P)$ is idempotent. Then, there is an isomorphism $\Phi:\insstable(D_1)\longrightarrow\insstable(D_2)$.
\end{cor}
\begin{proof}
If $\Spec(D_1),\Spec(D_2)$ are countable, then $D_1,D_2$ are min-scattered by Proposition \ref{prop:countable}. The claim now follows from Theorem \ref{teor:prufer-iso}.
\end{proof}

\section{Other versions}\label{sect:other}
Stable semistar operations are linked to two other structures on a ring: localizing systems and length functions.

A \emph{localizing system} on a domain $D$ is a set of ideals $\mathcal{F}$ such that:
\begin{itemize}
\item if $I\in\mathcal{F}$ and $I\subseteq J$, then $J\in\mathcal{F}$;
\item if $I\in\mathcal{F}$ and $(J:_DiD)\in\mathcal{F}$ for all $i\in I$, then $J\in\mathcal{F}$.
\end{itemize}
The map $\star\mapsto\mathcal{F}^\star:=\{I\mid 1\in I^\star\}$ establishes a bijective correspondence between the set of stable semistar operations and the set of all localizing systems, whose inverse is given by the map associating to $\mathcal{F}$ the semistar operation \cite[Section 2]{localizing-semistar}
\begin{equation*}
\star_\mathcal{F}:I\mapsto\bigcup_{J\in\mathcal{F}}(I:_DJ).
\end{equation*}

We say that a localizing system is \emph{radical} if, for every ideal $I$ such that $\rad(I)\in\mathcal{F}$, we have $I\in\mathcal{F}$. This notion corresponds exactly to radical semistar operations.
\begin{prop}
Let $\star$ be a stable semistar operation. Then, $\star$ is radical if and only if $\mathcal{F}^\star$ is a radical localizing system. 
\end{prop}
\begin{proof}
If $\star$ is radical and $\rad(I)\in\mathcal{F}^\star$, then $1\in\rad(I)^\star$ and thus $1\in I^\star$ by definition, so that $I\in\mathcal{F}$ and $\mathcal{F}^\star$ is radical. Conversely, if $\star$ is not radical there is an ideal $I$ such that $1\notin I^\star$ while $1\in\rad(I)^\star$: then, $\rad(I)\in\mathcal{F}^\star$ while $I\notin\mathcal{F}^\star$, so that $\mathcal{F}^\star$ is not radical.
\end{proof}

Therefore, the bijection between stable operations and localizing systems restricts to a bijection between $\insrad(D)$ and the set $\mathrm{LS}_{\mathrm{rad}}(D)$ of radical localizing systems; it follows that Theorems \ref{teor:insrad-iso} and \ref{teor:rcc-radspectral} can be expressed also in the terminology of localizing systems.

\medskip

A \emph{singular length function} on $D$ is a map $\ell$ from the set of all $D$-modules to $\{0,\infty\}$ such that:
\begin{itemize}
\item $\ell$ is additive on exact sequences;
\item for each module $M$, $\ell(M)$ is the supremum of $\ell(N)$, as $N$ ranges among the finitely generated submodules of $M$.
\end{itemize}
A singular length function is uniquely determined by its \emph{ideal colength} $\tau$, where $\tau$ is the map associated to each ideal $I$ the length $\ell(D/I)$ \cite[Proposition 3.3]{zanardo_length}. We denote by $\mathcal{L}_{\mathrm{sing}}(D)$ the set of singular length functions on $D$.

There is a bijective correspondence between localizing systems and singular length functions, where we associate to a localizing system $\mathcal{F}$ the colength \cite[Section 6]{length-funct}
\begin{equation*}
\tau_\mathcal{F}(I)=\begin{cases}
0 & \text{if~}I\in\mathcal{F},\\
\infty & \text{if~}I\notin\mathcal{F}.\\
\end{cases}
\end{equation*}

We say that a length function $\ell$ with associated colength $\tau$ is \emph{radical} if $\tau(I)=\tau(\rad(I))$ for every ideal $I$. This definition corresponds to radical semistar operations and radical localizing systems.
\begin{prop}
Let $\mathcal{F}$ be a localizing system. Then, $\mathcal{F}$ is radical if and only if the associated length function $\ell_\mathcal{F}$ is radical.
\end{prop}
\begin{proof}
If $I$ is an ideal, by definition $\tau(I)=0$ if and only if $I\in\mathcal{F}$. Therefore, $\tau(I)=\tau(\rad(I))=0$ if and only if $I,\rad(I)\in\mathcal{F}$, while $\tau(I)=\tau(\rad(I))=\infty$ if and only if $I,\rad(I)\notin\mathcal{F}$. The claim follows.
\end{proof}

Let $D$ be a Pr\"ufer domain. To every singular length function $\ell$ with associated colength $\tau$ we can associate the space
\begin{equation*}
\Sigma(\ell):=\{P\in\Spec(D)\mid \tau(Q)>0\text{~for some~}P\text{-primary ideal~}Q\},
\end{equation*}
and the length function
\begin{equation*}
\ell^\sharp:=\sum_{P\in\Sigma(\ell)}\ell\otimes D_P
\end{equation*}
(where $(\ell\otimes D_P)(M):=\ell(M\otimes D_P)$). Then, we get an analogue of Theorem \ref{teor:prufer-MinI}.
\begin{prop}\label{prop:prufer-length-MinI}
Let $D$ be a Pr\"ufer domain. The following are equivalent:
\begin{enumerate}[(i)]
\item\label{prop:prufer-length-Min:scat} $D$ is min-scattered;
\item\label{prop:prufer-length-Min:stdstable} $\ell=\ell^\sharp$ for every singular length function $\ell$.
\end{enumerate}
\end{prop}
\begin{proof}
Let $\Phi$ be the isomorphism between $\insstable(D)$ and $\mathcal{L}_{\mathrm{sing}}(D)$ obtained composing the bijections of the two with the set of localizing systems. The claim follows from Theorem \ref{teor:prufer-MinI} and the fact that $\Phi(\stdstable{\star})=\Phi(\star)^\sharp$ \cite[Proposition 6.8]{length-funct}.
\end{proof}

As a consequence, we also obtain a version of Theorem \ref{teor:prufer-iso} (compare with \cite[Theorem 5.12]{length-funct}).
\begin{teor}\label{teor:prufer-iso-length}
Let $D_1,D_2$ be Pr\"ufer domain. Suppose that there is a homeomoprhism $\phi:\Spec(D_1)\longrightarrow\Spec(D_2)$ such that a prime ideal $P$ is idempotent if and only if $\phi(P)$ is idempotent. If $D_1$ is min-scattered, then there is an isomorphism $\Phi:\mathcal{L}_{\mathrm{sing}}(D_1)\longrightarrow\mathcal{L}_{\mathrm{sing}}(D_2)$.
\end{teor}

To conclude, we express \cite[Corollary 7.5]{almded-radfact} in the terminology of this paper; see \cite{almded-radfact} for the definition of SP-scattered domain.
\begin{prop}
Let $D$ be an SP-scattered domain. Then, every stable semistar operation on $D$ is radical.
\end{prop}

\bibliographystyle{plain}
\bibliography{/bib/miei,/bib/articoli,/bib/libri}

\begin{thebibliography}{10}

\bibitem{anderson_overrings_1988}
D.~D. Anderson.
\newblock Star-operations induced by overrings.
\newblock {\em Comm. Algebra}, 16(12):2535--2553, 1988.

\bibitem{anderson_intersections_2005}
D.~D. Anderson and Sharon~M. Clarke.
\newblock Star-operations that distribute over finite intersections.
\newblock {\em Comm. Algebra}, 33(7):2263--2274, 2005.

\bibitem{anderson_two_2000}
D.~D. Anderson and Sylvia~J. Cook.
\newblock Two star-operations and their induced lattices.
\newblock {\em Comm. Algebra}, 28(5):2461--2475, 2000.

\bibitem{spectralspaces-libro}
Max Dickmann, Niels Schwartz, and Marcus Tressl.
\newblock {\em Spectral spaces}, volume~35 of {\em New Mathematical
  Monographs}.
\newblock Cambridge University Press, Cambridge, 2019.

\bibitem{fifolo_transactions}
Carmelo~A. Finocchiaro, Marco Fontana, and K.~Alan Loper.
\newblock The constructible topology on spaces of valuation domains.
\newblock {\em Trans. Amer. Math. Soc.}, 365(12):6199--6216, 2013.

\bibitem{spettrali-eab}
Carmelo~A. Finocchiaro, Marco Fontana, and Dario Spirito.
\newblock Spectral spaces of semistar operations.
\newblock {\em J. Pure Appl. Algebra}, 220(8):2897--2913, 2016.

\bibitem{topological-cons}
Carmelo~A. Finocchiaro and Dario Spirito.
\newblock Some topological considerations on semistar operations.
\newblock {\em J. Algebra}, 409:199--218, 2014.

\bibitem{localizing-semistar}
Marco Fontana and James~A. Huckaba.
\newblock Localizing systems and semistar operations.
\newblock In {\em Non-{N}oetherian commutative ring theory}, volume 520 of {\em
  Math. Appl.}, pages 169--197. Kluwer Acad. Publ., Dordrecht, 2000.

\bibitem{fontana_loper-eab}
Marco Fontana and K.~Alan Loper.
\newblock Cancellation properties in ideal systems: a classification of
  {$\mathtt{e.a.b.}$} semistar operations.
\newblock {\em J. Pure Appl. Algebra}, 213(11):2095--2103, 2009.

\bibitem{gilmer}
Robert Gilmer.
\newblock {\em Multiplicative {I}deal {T}heory}.
\newblock Marcel Dekker Inc., New York, 1972.
\newblock Pure and Applied Mathematics, No. 12.

\bibitem{krull_breitage_I-II}
Wolfgang Krull.
\newblock Beitr\"age zur {A}rithmetik kommutativer {I}ntegrit\"atsbereiche
  i-ii.
\newblock {\em Math. Z.}, 41(1):545--577; 665--679, 1936.

\bibitem{mazur-sierp-numerabili}
Stefan Mazurkiewicz and Wac{\l}aw Sierpi\'{n}ski.
\newblock Contribution \`{a} la topologie des ensembles d\'enombrables.
\newblock {\em Fundam. Math.}, 1:17--27, 1920.

\bibitem{okabe-matsuda}
Akira Okabe and Ry{\=u}ki Matsuda.
\newblock Semistar-operations on integral domains.
\newblock {\em Math. J. Toyama Univ.}, 17:1--21, 1994.

\bibitem{stable_prufer}
Dario Spirito.
\newblock Towards a classification of stable semistar operations on a
  {P}r\"ufer domain.
\newblock {\em Comm. Algebra}, 46(4):1831--1842, 2018.

\bibitem{non-ft}
Dario Spirito.
\newblock The {Z}ariski topology on sets of semistar operations without
  finite-type assumptions.
\newblock {\em J. Algebra}, 513:27--49, 2018.

\bibitem{localizzazioni}
Dario Spirito.
\newblock Topological properties of localizations, flat overrings and
  sublocalizations.
\newblock {\em J. Pure Appl. Algebra}, 223(3):1322--1336, 2019.

\bibitem{length-funct}
Dario Spirito.
\newblock Decomposition and classification of length functions.
\newblock {\em Forum Math.}, 32(5):1109--1129, 2020.

\bibitem{jaff-derived}
Dario Spirito.
\newblock The derived sequence of a pre-{J}affard sequence.
\newblock {\em Mediterr. J. Math.}, 19(4):146, 2022.

\bibitem{almded-radfact}
Dario Spirito.
\newblock Almost {D}edekind domains without radical factorization.
\newblock submitted.

\bibitem{PicInt}
Dario Spirito.
\newblock Localizations of integer-valued polynomials and of their {P}icard
  group.
\newblock submitted.

\bibitem{zanardo_length}
Paolo Zanardo.
\newblock Multiplicative invariants and length functions over valuation
  domains.
\newblock {\em J. Commut. Algebra}, 3(4):561--587, 2011.

\end{thebibliography}
\end{document}